\documentclass[12pt,reqno]{amsart}
\usepackage[nospace,noadjust]{cite}

\usepackage{verbatim,amscd,amssymb}
\usepackage{graphicx}
\usepackage{amsmath}
	\usepackage{amsmath}
\usepackage{amsfonts}
\usepackage{amssymb}
\usepackage{inputenc}

\usepackage{enumitem}
\usepackage{float}

\usepackage[active]{srcltx}

\graphicspath{ {./images/} }

\usepackage{fontenc}

\usepackage{float}

\newtheorem{theorem}{Theorem}[section]
\newtheorem{lemma}[theorem]{Lemma}
\newtheorem{cor}[theorem]{Corollary}

\theoremstyle{definition}
\newtheorem{definition}[theorem]{Definition}
\newtheorem{example}[theorem]{Example}

\theoremstyle{remark}

\theoremstyle{assumption}

\numberwithin{equation}{section}

\newcommand{\uc}{\mathbb{S}}

\newcommand{\sha}{\succ\mkern-14mu_s\;}

\begin{document}
	
	\date{\today}
	
	\title[A new invariant for a cycle of an interval map]{A new invariant for a cycle of an interval map}

	\author{Sourav Bhattacharya}

	\address[Dr. Sourav Bhattacharya]
	{Department of Mathematics, Visvesvaraya National Institute Of Technology Nagpur, 
 Nagpur, Maharashtra 440010, India}
	
	\email{souravbhattacharya@mth.vnit.ac.in}
	
	

	\begin{abstract}
		We \emph{propose} a new \emph{invariant} for a \emph{cycle} of an \emph{interval map} $f:[0,1] \to [0,1]$,  called its \emph{unfolding number}. 
	\end{abstract}
	
	\maketitle

\section{Introduction}\label{introduction}
In his 1964 legendary paper \cite{S} (see \cite{shatr} for English Translation), O.M. Sharkovsky unveiled an elegant \emph{rule} governing the \emph{co-existence} of \emph{periods} of \emph{periodic orbits} (also known as \emph{cycles}) of a \emph{continuous interval map}. He introduced the following \emph{order} among all \emph{natural numbers}: $$3\sha 5\sha 7\sha\dots\sha 2\cdot3\sha 2\cdot5\sha 2\cdot7 \sha \dots $$
$$
\sha\dots 2^2\cdot3\sha 2^2\cdot5\sha 2^2\cdot7\sha\dots\sha 8\sha
4\sha 2\sha 1$$ and proved the following theorem: 

\begin{theorem}[\cite{S,shatr}]\label{shar}
	If $m, n \in \mathbb{N}$ with $m \sha n$,  then if a  continuous mapping $f: [0,1] \to [0,1]$ has a cycle of period $m$, then $f$ must also have a cycle of period $n$.
\end{theorem}

 Theorem \ref{shar} reveals an intricate and \emph{concealed combinatorial} framework that \emph{governs} the  \emph{disposition} of \emph{cycles} of a \emph{continuous interval} map and led to the inception of a new direction of research in the theory of \emph{dynamical systems} called \emph{combinatorial dynamics}. 

Future research stemming from Theorem \ref{shar} could branch out in numerous directions. For instance, one might \emph{replace} the \emph{interval} with a more \emph{general space} or \emph{replace} \emph{periodic orbits} with more \emph{general orbits} (e.g., \emph{homoclinic trajectories}). Additionally, one could strive to discover a more ``\emph{refined}" \emph{rule} for the \emph{coexistence} of \emph{periodic orbits} in a given \emph{continuous interval map} than that provided by Theorem \ref{shar}. For an extensive list of references and a comprehensive survey of the \emph{evolution} of Theorem \ref{shar}, its historical context, subsequent advancements, and implications in the \emph{theory of dynamical systems}, one could refer to the books \cite{alm00} and \cite{BOM} and the paper \cite{BME}.

In this paper, we pursue the latter mentioned approach for \emph{refining} Theorem \ref{shar}. Let us discuss in details. Observe that Theorem \ref{shar} elucidates how certain ``\emph{type}"  of \emph{dynamical behaviour}  (here \emph{period}) once exhibited by a map, \emph{forces} some other ``\emph{type}''  (\emph{period}) and thus, the question of ``\emph{co-existence}"  among ``\emph{types}'' becomes a question of ``\emph{forcing}" among ``\emph{types}''.  However, ``\emph{periods}" offer only a coarse depiction of a \emph{cycle}, as numerous \emph{cycles} share the same \emph{period}. A more nuanced approach to characterizing \emph{cycles}, involves examining their \emph{cyclic permutations}: \emph{cyclic
	permutation} we get when we look at how the map acts on the points
of the \emph{cycle}, ordered from the \emph{left to the right}; we \emph{identify}  \emph{cycles} having the \emph{same} underlying \emph{cyclic permutations} and call the corresponding \emph{equivalence classes}: \emph{patterns}.

   As it turns out (see \cite{Ba}), the \emph{rule} for \emph{co-existence} using this ``\emph{type}" is rather too \emph{detailed} and doesn't allow for a transparent description.  This
motivates one to look for a, middle-of-the-road way of describing
cycles: \emph{rotation theory}, a way not as \emph{crude} as \emph{periods} but not as fine as \emph{permutations},
which would still allow for a \emph{transparent description}.

 	The concept of \emph{rotation numbers} was  historically first introduced by Poincar\'e
 	for \emph{circle homeomorphisms} (See  \cite{poi}).  It was extended to \emph{circle maps} of
 	\emph{degree one} by Newhouse, Palis and Takens (see \cite{npt83}), and then
 	studied, e.g., in \cite{bgmy80, ito81, cgt84, mis82, mis89, almm88}. In a general setting,  \emph{rotation numbers} can be defined as follows: Let $X$ be a \emph{compact metric space} with a Borel $\sigma$-algebra, $\phi:X\to\mathbb R$ be a \emph{bounded measurable function} (often called an \emph{observable}) and 
 $f:X\to X$ be a \emph{continuous} map. Then for any $x \in X$ the set
 $I_{f,\phi}(x)$, of all \emph{sub-sequential limits} of the sequence
 $ \left \{ {\frac1n} \sum^{n-1}_{i=0}\phi(f^i(x)) \right \}$ is called the {\it
 	$\phi$-rotation set} of $x$.  If $I_{f,\phi}(x)=\{\rho_\phi(x)\}$ is a singleton, then the
 number $\rho_\phi(x)$ is called the {\it $\phi$-rotation number} of
 $x$. It is easy to see that the $\phi$-\emph{rotation set}, $I_{f,\phi}(x)$ is a closed
interval for all $x \in X$. The \emph{union} of all $\phi$-\emph{rotation sets} of all points of $X$ is called
the \emph{$\phi$-rotation set} of the map $f$ and is denoted by
$I_f(\phi)$. If $x$ is an $f$-\emph{periodic point} of period $n$,  then its {\it $\phi$-rotation number} 
$\rho_\phi(x)$ is well-defined, and a related concept of the
\emph{$\phi$-rotation pair} of $x$ can be introduced: the pair
$({\frac1n}\sum^{n-1}_{i=0}\phi(f^i(x)), n)$ is the
\emph{$\phi$-rotation pair} of $x$. 

	 The use of \emph{rotation theory} in the study of \emph{dynamical systems} is useful when our \emph{rotation theory} is \emph{complete}. This means that: 
	 
	 \begin{enumerate}
	 	\item  \emph{rotation set} is \emph{convex};
	 	\item  \emph{rotation number} of a \emph{cycle} of \emph{period} $n$ is of the form $\frac{m}{n}$ where $m,n \in \mathbb{N}$;
	 	\item  if $\frac{m}{n}$ belongs to the \emph{interior} of the \emph{rotation set}, then there exists a \emph{cycle} of \emph{period} $kn$ and \emph{rotation pair} $(km,kn)$ for every $k \in \mathbb{N}$.  
	 	
	 \end{enumerate}

	Another way of looking at \emph{rotation theory} is through \emph{ergodic invariant measures} (\cite{dgs76}, \cite{alm00}).   If $P$ is a \emph{cycle} of $f: [0,1] \to [0,1]$, we can look at the \emph{ergodic measure} $\mu_P$, \emph{equidistributed} on $P$, that is, such that $\mu_P(\{ x\}) = \frac{1}{n}$ for all $ x\in P$, where $n$ is the \emph{period} of $P$.  Let us call $\mu_P$, a \emph{CO-measure}
\cite{dgs76} (comes from ``\textbf{c}losed \textbf{o}rbit''). Recall that for \emph{Borel
measures} on \emph{compact spaces} one normally considers their \emph{weak}
topology defined by the \emph{continuous functions} \cite{dgs76}. The following result follows: 

\begin{theorem}[\cite{ blo95a}]\label{t:codense} Suppose that
	$f: [0,1] \to [0,1] $ is a continuous interval map or a circle map with non-empty
	set of periodic points. Then any invariant probability measure $\mu$
	for whom there exists a point $x$ with $\mu(\omega_f(x))=1$,  where $\omega_f(x)$ denotes omega limit set of $x$ under $f$ can be
	approximated by CO-measures arbitrary well. In particular, CO-measures
	are dense in the set of all ergodic invariant measures of $f$.
\end{theorem}

If the \emph{dynamics} of $f$ and the  \emph{observable} $\phi$ are related we can further deduce additional informations about $\phi$-\emph{rotation sets}, for instance, in the case of \emph{circle degree one} case (see \cite{mis82}, \cite{alm00}).  Consider the circle $\uc = \mathbb{R} / \mathbb{Z}$ with the \emph{natural projection} $\pi: \mathbb{R} \to \uc$. If $f:\uc\to \uc$ is \emph{continuous}, then
	there is a \emph{continuous map} $F_f:\mathbb{R}\to \mathbb{R}$ such that $F_f
	\circ \pi=\pi\circ f$. Such a map $F_f$ is called a \emph{lifting} of $f$. It
	is unique up to translation by an integer. An integer $d$ with $F_f(X+1)=F_f(X)+d$
	for all $X\in \mathbb{R}$ is called the \emph{degree} of the map $f$ and is
	independent of the choice of $F$. \emph{Conversely}, maps $G:\mathbb R\to \mathbb R$ such that $G(X+1)=G(X)+d$ for every real $X$ are said to be \emph{maps of the real line of degree $d$}, can be defined independently, and are \emph{semi-conjugate} by the same map $\pi: \mathbb{R} \to \uc$ to the \emph{circle maps of degree} $d$. In this paper we consider both \emph{maps of the circle} and \emph{maps of the real line} \emph{of degree one}. Denote by $\mathcal{L}_1$ the set of all \emph{liftings} of \emph{continuous
	degree one} self-mappings of $\uc$ endowed with the \emph{sup norm}.

	Given, a \emph{continuous map},   $f : \uc \to \uc$,   consider its \emph{lifting} $F_f \in \mathcal{L}_1$. Define an \emph{observable} $\phi_f: \uc \to \mathbb R$ so that
$\phi_f(x)=F_f(X)-X$ for any $x \in \uc$ and  $X\in \pi^{-1}(x) \in \mathbb{R}$; then $\phi_f$ is
well-defined, the \emph{classical rotation set} of a point $x \in \uc $ is simply the $\phi_f$ -\emph{rotation set} $I_{f,\phi_f}(x)=I_f(x)$. It is easy to see that $I_f(x) = \left \{ \overline{\rho_{F_f}}(X),\underline{\rho_{F_f}}(X) \right  \}$ where $\overline{\rho_{F_f}}(X) = \displaystyle  \lim \sup_{ n \to \infty} \frac{F_f^{n}(X) - X}{n}$ and $\underline{\rho_{F_f}}(X) = \displaystyle  \lim \inf_{ n \to \infty} \frac{F_f^{n}(X) - X}{n}$.  The \emph{numbers}: $\overline{\rho_{F_f}}(X)$ and $ \underline{\rho_{F_f}}(X) $ are respectively called \emph{upper} and \emph{lower} \emph{rotation numbers} of $X \in \mathbb{R}$ under $F_f: \mathbb{R} \to \mathbb{R}$ (equivalently of $x \in \uc$ under $f: \uc \to \uc$). If $\overline{\rho_{F_f}}(X) = \underline{\rho_{F_f}}(X) = \rho_{F_f}(X)$, we write $\rho_{F_f}(X)$, the \emph{rotation number} of $X	$ under $F_f$ (equivalently of $x$ under $f: \uc \to \uc$).

Suppose $f^{n}(y) = y$, $y \in \uc$, that is, $y$ is a \emph{periodic point} of $f: \uc \to \uc$ of \emph{period} $n$. Then, for any $Y \in \pi^{-1}(y)$, $F_f^{n}(Y) = Y + m$ for some $ m \in \mathbb{N}$ and hence $\overline{\rho_{F_f}}(Y) = \underline{\rho_{F_f}}(Y) =  \frac{m}{n}$. We call $  \frac{m}{n}$, the \emph{rotation number} of $Y \in \mathbb{R}$ under $F_f: \mathbb{R} \to \mathbb{R}$ (equivalently of $y \in \uc$ under $f : \uc \to \uc$)  and $(m,n)$ the \emph{rotation pair} of $Y \in \mathbb{R}$ under $F_f: \mathbb{R} \to \mathbb{R}$  (equivalently of $y \in \uc$ under $f : \uc \to \uc$).

 	\emph{Rotation pairs} can be represented in an interesting way using the notations used in  \cite{BMR}. We  represent the \emph{rotation pair} $rp(x) = (mp, mq)$ with $p,q$ coprime and $m$ being some natural number as a pair $ (t, m)$ where $ t = \frac{p}{q}$. We call the later pair a \emph{modified-rotation pair}(\emph{mrp}) and write $mrp(x) = (t, m)$. We then think of the real line with a \emph{prong} attached at each rational point and the set $\mathbb{N}
	\cup \{2^\infty\} $ marked on this prong in the \emph{Sharkovsky ordering} $ \sha$ with $1$ \emph{closest} to the \emph{real line} and $3$ \emph{farthest} from it. All points of the \emph{real line} are marked $0$; at irrational points we can think of degenerate prongs with only $0$ on them. The union of all \emph{prongs} and the \emph{real line} is denoted by $\mathbb{M}$.  Thus, a \emph{modified rotation pair} $(t,m)$ corresponds to the specific element of $ \mathbb{M}$, namely to the number $m$ on the \emph{prong} attached at $t$.  However, no \emph{rotation pair} corresponds to $(t,2^{\infty})$ or to $(t,0)$. Then, for  $(t_1, m_1) $ and  $  (t_2,m_2)$, in $\mathbb{M}$, the \emph{convex hull}	 $[(t_1, m_1), $ $  (t_2,m_2)]$  consists of all \emph{modified rotation pairs} $(t,m)$ with $t$ strictly between $t_1$ and $t_2$ or $t=t_i$ and $m \in Sh(m_i)$ for $ i=1,2$ (See Figure \ref{convex_hull}).

	Let $F \in \mathcal{L}_1$. Let $mrp(F)$ be the set of \emph{modified rotation pairs} of all \emph{lifted cycles} of $F$.  Clearly, $mrp(F) \subset \mathbb{M}$. The following theorem holds:
	
	\begin{theorem}[\cite{mis82}]\label{circle:maps}
		Let $F \in \mathcal{L}_1$. Then there are elements $(t_1,m_1)$ and $(t_2,m_2)$ of $\mathbb{M}$ such that $mrp(F) = [(t_1,m_1), (t_2,m_2)]$ and if $t_i$ is rational, then $m_i \neq 0$ for $i=1,2$. Moreover, for any set of the above form there exists $F \in \mathcal{L}_1$ with $mrp(F)$ equal to this set.
	\end{theorem}

	\begin{figure}[H]
		\caption{\emph{Pictorial representation} of the \emph{convex hull} (\emph{shown in blue}) $[(t_1,m_1), (t_2,m_2)]$ of the \emph{modified rotation pairs} $(t_1,m_1)$ and $(t_2,m_2)$ where $t_1,t_2 \in \mathbb{Q}$ and $m_1,m_2 \in \mathbb{N}$}
		\centering
		\includegraphics[width=0.5\textwidth]{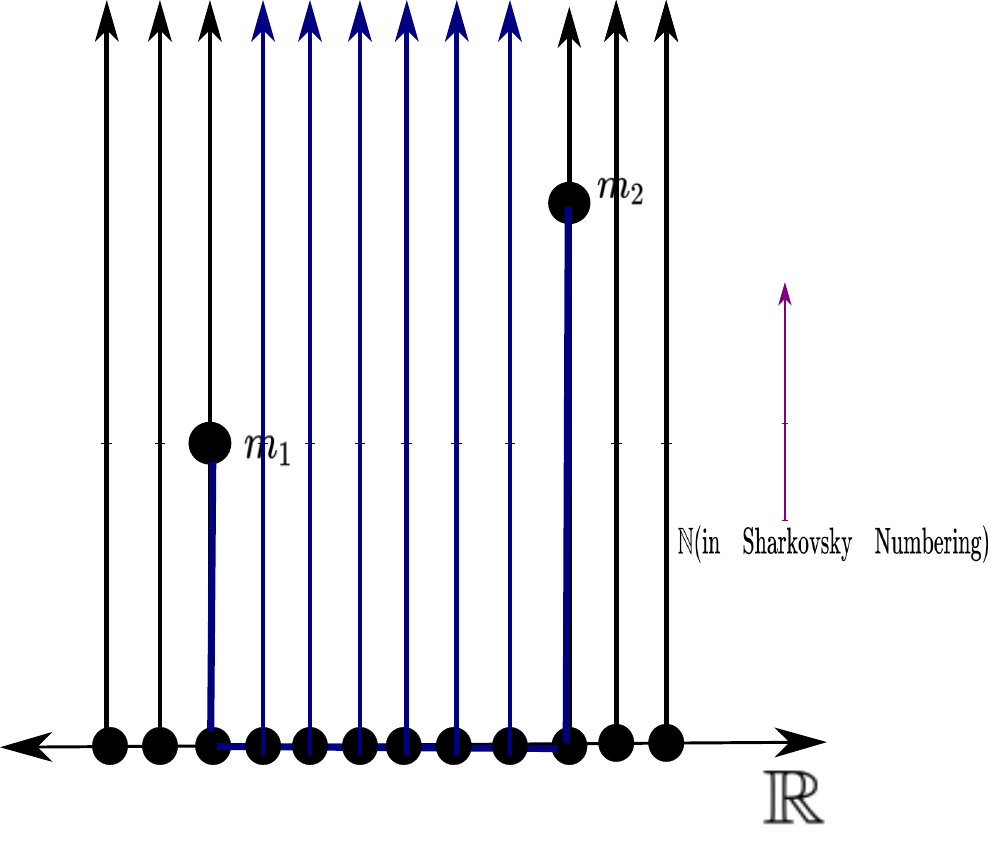}
		\label{convex_hull}
	\end{figure}
	
	Results similar to Theorem \ref{circle:maps} were also proven for maps of intervals in \cite{BM1} when a version of \emph{rotation numbers} for \emph{interval maps} called  \emph{over-rotation numbers} was introduced using a special \emph{observable} $\chi$. Let $f:[0,1]\to [0,1]$ be a \emph{continuous interval map}, $Per(f)$ be its set
	of \emph{periodic points}, and $Fix(f)$ be its set of \emph{fixed points}. It is easy
	to see that if $Per(f)=Fix(f)$,  then the \emph{omega limit set}, $\omega(y)$ is a \emph{fixed point} for
	any $y$. So, we can assume that $Per(f)\neq Fix(f)$. Now, define:

	$$ \displaystyle   \chi(x) =\begin{cases} \frac{1}{2} &\text{if $(f(x)-x)(f^2(x)-f(x))\leqslant  0$,}\\ {0}&\text{if
			$(f(x)-x)(f^2(x)-f(x))>0$.}\end{cases}$$

	For any non-fixed periodic point $y$ of period $p(y)$, the integer
	$l(y)= \displaystyle  \sum^{p(y)-1}_{i=0}\chi(f^i(y))$ is at most $\frac{p(y)}{2}$ and is the
	same for all points from the orbit of $y$. The quantity $\frac{l(y)}{p(y)} = orn(y)$ is called the \emph{over-rotation number} of $y$ and the pair $orp(y)=(l(y),
	p(y))$ is called the {\it over-rotation pair} of $y$. In an \emph{over-rotation pair} $(p,q)$, $p$ and $q$ are integers and $0<\frac{p}{q}
	\leqslant \frac{1}{2} $.  	Like before, we can transform all \emph{over-rotation pairs} of \emph{cycles} of a given map $f$ into \emph{modified over-rotation pairs} of $f$ and denote the \emph{modified over-rotation pair} of a cycle $P$ and the set of all \emph{modified over-rotation pairs} of \emph{cycles} of $f$ by $mop(P)$ and $mop(f)$ respectively. Then, again $ mop(f) \subset \mathbb{M}$ and the following theorem holds:

	\begin{theorem}[\cite{BM1}]\label{modified:over:rot}
		If $ f :[0,1] \to [0,1]$ is a continuous interval map with a non-fixed periodic point, then $ mop(f) = [(t_1, m_1), (\frac{1}{2}, 3) ]$ for some $(t_1, m_1) \in \mathbb{M}$. Moreover, for every $ (t_1, m_1) \in \mathbb{M}$, there exists a continuous map $f : [0,1] \to [0,1]$ with $ mop(f) = [(t_1, m_1), (\frac{1}{2}, 3)]$. 
	\end{theorem}
	
	Given an interval map $f$, denote by $I_f$, the \emph{closure of the union of over-rotation numbers} of $f$-\emph{periodic points}, and call $I_f$ the \emph{over-rotation interval} of $f$. By Theorem \ref{shar}, any map $f$ with a \emph{non-fixed periodic
		points} has a \emph{cycle} of \emph{period} 2 and such as cycle has \emph{over-rotation number} $\frac{1}{2}$;
	by Theorem \ref{modified:over:rot}, if $\rho(P)=\frac{p}{q}$ for a cycle $P$ then
	$[\frac{p}{q}, \frac12]\subset I_f$; hence for any interval map $f$ there exists
	a number $r_f$,  $0\leqslant r_f < \frac12,$ such that $I_f=[r_f, \frac12]$. 
		A \emph{cycle} $P$ is called \emph{over-twist} if it does not \emph{force} other \emph{cycles} with the same \emph{over-rotation number}. They are the analogues of \emph{twist cycles} for \emph{interval maps}.  By \cite{BM1} and by \emph{properties of forcing relation}, for any
	$\frac{p}{q}\in (r_f, \frac{1}{2})$, $f$ has an \emph{over-twist cycle}
	of \emph{over-rotation number} $\frac{p}{q}$. We call a \emph{pattern} $\pi$, an \emph{over-twist pattern} if any \emph{cycle} which \emph{exhibits} it is an \emph{over-twist} cycle. Now we discuss our plans for the paper.

In the papers  \cite{BB2} and \cite{BS}, an \emph{algorithm} to find out \emph{over-rotation interval} of an \emph{interval map} of \emph{modality} less than or equal to $2$ was concocted and the complete description of the \emph{dynamics} of \emph{unimodal} and \emph{bimodal} \emph{over-twist patterns} was corroborated.  However, the task of determining the \emph{over-rotation interval} for an \emph{interval map} $f:[0,1] \to [0,1]$  with \emph{modality greater than or equal to} $3$, along with elucidating \emph{over-twist patterns} of the same \emph{modality}, presents significant challenges.

This maneuvered us to develop a new \emph{invariant} for a  \emph{cycle} of an \emph{interval} map,   amalgamating the essence  of \emph{over-rotation theory} of \emph{interval maps} and \emph{rotation theory} of \emph{degree one circle maps}. We call this new invariant:   \emph{unfolding numbers}.  The basic idea is as follows. Given a \emph{continuous map} $f: [0,1] \to [0,1]$, we construct the map $g_f :[0, \frac{1}{2}] \to [0, \frac{1}{2}]$ which is \emph{conjugate} to the map $f$ via the  \emph{linear contraction} $\lambda: [0,1] \to [0, \frac{1}{2}]$ defined by $\lambda(x) = \frac{x}{2}$. We call $g_f$, the \emph{miniature model map} corresponding to the map $f$.   We \emph{fold} the \emph{unit square} $[0,\frac{1}{2}] \times [0,\frac{1}{2}]$ containing the \emph{graph} of $g_f$ \emph{vertically} across the line $y =\frac{1}{2}$.  This yields a map $p_f$ in $[0,\frac{1}{2}] \times [\frac{1}{2}, 1]$. The \emph{unit square} $[0,\frac{1}{2}] \times [\frac{1}{2}, 1]$ is now \emph{folded} \emph{horizontally} across $x=\frac{1}{2}$ to obtain a map $q_f$ in $[\frac{1}{2},1] \times [\frac{1}{2},1] $.  Finally, the \emph{unit square}  $[\frac{1}{2},1] \times [\frac{1}{2},1]$ is  \emph{folded} \emph{vertically} across $y=1$ to obtain a map $r_f$ in $[\frac{1}{2},1] \times [1, \frac{3}{2}] $. Selecting \emph{parts} of \emph{graph} of the \emph{maps}: $g_f$ , $p_f$, $q_f$ and $r_f$  yields a \emph{degree one map} $F_f : \mathbb{R} \to \mathbb{R}$ of the real line (\emph{for details see Section} \ref{first:main:section}). The \emph{unfolding number} of a \emph{periodic orbit} $P$ of the map $f: [0,1] \to [0,1]$ is defined to  to be \emph{rotation number} of the \emph{lifted periodic orbit} $P'$ of $F_f$ corresponding to $P$. We begin study of \emph{unfolding numbers} for \emph{interval maps} in this paper.

	The paper is divided into $8$ \emph{sections}. The  manuscript ``\emph{unfolds}" as follows:

\begin{enumerate}
	\item Section \ref{introduction}, serves as the \emph{introductory segment} of the paper, providing an initial \emph{overview} and context for the subsequent content. 
\item Section \ref{preliminaries}, lays the groundwork with essential \emph{preliminaries}.

\item In Section \ref{first:main:section} ,we define \emph{unfolding numbers} and enact the \emph{consistency} of the \emph{theory}.

\item In Section \ref{unfolding:interval:section}, we portray that, that the \emph{closure} of the set of \emph{unfolding numbers} of \emph{periodic orbits} of a given \emph{interval map} $f$ forms an interval of the form $Unf(f) = [u_f, \frac{1}{2}]$ for some $u_f \geqslant 0$, called its \emph{unfolding interval}. The \emph{existence} of a \emph{point} $z \in [0, 1]$ such that the map $f$ \emph{restricted on} the \emph{omega limit set} $Z_f = \omega_f(z)$ of $z$ is \emph{strictly ergodic},  and such that for any point $y \in Z_f$, the \emph{unfolding number} of $y$ under $f$ coincides with $u_f$ is also revealed.

\item It is a natural question to ask: can we characterize the \emph{dynamics} of a \emph{periodic orbit} of $f$ given its \emph{unfolding number}? For this, at the outset it is vital to see whether \emph{unfolding number} of $P$ can be read directly from the \emph{combinatorics} of $P$ without the need to construct the map $F_f$.  We take precisely these \emph{questions} in Section \ref{reading:unfolding:comb}.

\item In Section \ref{forced:unfolding:interval}, we introduce the idea of \emph{forced unfolding interval}. We elucidate \emph{patterns} for whom the \emph{ forced unfolding interval} is $[0, \frac{1}{2}]$.

\item In Section \ref{comparison:over:numbers}, we forge \emph{comparison} between \emph{unfolding numbers} and \emph{over-rotation numbers}: we show  that while these two \emph{rotation numbers} are \emph{disparate} in flavor, there exists \emph{special patterns} $\pi_s$ called \emph{sheer patterns}, for whom \emph{unfolding numbers} equals \emph{over-rotation numbers}!

\end{enumerate}

\bigskip 

From the \emph{form} of \emph{unfolding interval} $Unf(f)$,  of a given \emph{interval map} $f$, it is evident that  a \emph{single number}, $u_f$, encapsulates concealed information about the \emph{dynamical complexity} of $f$. Further $u_f$ can be computed using a straightforward algorithm for interval maps of any \emph{modality}. Thus, \emph{unfolding numbers} provides a ``new" effective means of computing the \emph{dynamical complexity} of an \emph{interval map}  which is \emph{invariant} under \emph{topological conjugacy}.

\subsection*{Acknowledgment} The author is thankful to his teacher, Professor Alexander Blokh from University of Alabama at Birmingham, USA with whom the problem was first discussed.

\section{Preliminaries}\label{preliminaries}

\subsection{Notations}\label{subsec:notations}

In this paper, we specifically consider \emph{continuous maps} $f: [0,1] \to [0,1]$ characterized by a \emph{unique point} of \emph{absolute maximum}, denoted by $M_f$, and a \emph{unique point} of \emph{absolute minimum}, denoted by $m_f$. Moreover, for any \emph{periodic orbit} $P$ of $f$, its \emph{leftmost endpoint} is represented by $le(P)$ and its \emph{right end point} is denoted by $ri(P)$. Furthermore, for any \emph{point} $x$ and a map $g$, the \emph{trajectory} of $x$ under $g$, that is,  the set $\{ x, g(x), g^2(x), \dots\}$ is symbolized as $T_g(x)$. Clearly, if $T_g(x)$ is \emph{finite} with $q$ \emph{elements}, $x$ is \emph{periodic} under $g$ with  \emph{period} $q$ and  $P = T_g(x)$ is its \emph{periodic orbit}. For a set $A$, let $\overline{A}$ denotes its \emph{closure}. Additionally, given a map $f : [0,1] \to [0,1]$, $g_f : [0, \frac{1}{2}] \to [0, \frac{1}{2}]$ and $F_f : \mathbb{R} \to \mathbb{R}$,  respectively signify \emph{miniature model map} and \emph{heaved map} corresponding to the map $f$ (\emph{as defined in} Section \ref{first:main:section}).

\subsection{Forcing among patterns}

\begin{definition}
	A \emph{pattern} $A$ \emph{forces} \emph{pattern} $B$ if every continuous map having a
	\emph{cycle} of \emph{pattern} $A$ has a \emph{cycle} of \emph{pattern} $B$. By \cite{Ba}, \emph{forcing among patterns}
	is a \emph{partial ordering}.
\end{definition}
 A useful \emph{algorithm} allows one to describe all \emph{patterns} \emph{forced} by a given \emph{pattern} $A$. Namely, consider a \emph{cycle} $P$ of \emph{pattern} $A$; assume that \emph{leftmost} and \emph{rightmost} \emph{points} of $P$ are $a$ and 
$b$ respectively. Every \emph{component} of $[a, b]\setminus P$ is said to be a
\emph{$P$-basic interval}. \emph{Extend} the map from $P$ to the interval $[a,
b]$ by defining it \emph{linearly} on each $P$-\emph{basic interval} and call the
resulting map $f_P$,  the \emph{$P$-linear map}. The following result follows:

\begin{theorem}[\cite{Ba,alm00}]
	
	The patterns of all cycles of $f_P$ are exactly the patterns forced by the pattern of $P$
	\end{theorem}

\subsection{Loops of intervals}\label{loops:intervals}

\begin{definition}
	 Let $\{I_n\}$, $n\geqslant 0$ be a \emph{sequence} of  \emph{closed intervals} such that $f(I_j) \supset
	I_{j+1}$ for $ j \geqslant 0 $; then we say that $\{I_n\}$, $n\geqslant 0$ is an
	\emph{$f-$chain} or simply a \emph{chain} of intervals. If a \emph{finite
		chain of intervals} $\alpha = \{I_0, I_1, \dots , I_{k-1} \}$ is such that $f(I_{k-1})
	\supset I_0 $, then we call $\alpha = \{I_0, I_1, \dots , I_{k-1} \}$, an
	$f$-\emph{loop} or simply a \emph{loop} of  \emph{intervals}.
	
\end{definition}

\begin{lemma}[\cite{alm00}]\label{ALM2}
Let $\alpha = \{I_0, I_1, \dots , I_{k-1} \}$ be a loop of intervals. Then
		there is a periodic point $x \in I_0$ such that $f^j(x) \in I_j $ for $
		0 \leqslant  j \leqslant  k-1 $ and $ f^k(x)=x$.  $Q = \{ x, f(x), f^2(x), \dots f^{k-1}(x) \}$ is called  the periodic orbit associated with the loop $\alpha$. 

\end{lemma}

\subsection{Results on Degree one circle maps}\label{subsec:results:degree:one:maps}

Let's examine some properties of  \emph{degree one circle maps} that are pertinent to our discussion in this paper. Let $\mathcal{L}_1$ represents the \emph{collection} of all \emph{liftings} of \emph{continuous degree one self-mappings} of $\uc$ under the \emph{sup norm}, while $ \mathcal{L}_1' \subset \mathcal{L}_1$  consists of \emph{non-decreasing elements} of $\mathcal{L}_1$. Subsequently, the ensuing result follows: 

\begin{theorem}[\cite{alm00,mis82}]\label{non:decreasing:exists}
	If $F_f \in \mathcal{L}_1'$ is a lifting of a circle map $f: \uc \to \uc$ then $\rho_{F_f}(x)$ exists for all $x \in \mathbb{R}$ and is independent of $x$. Moreover, it is rational if and only if $f: \uc \to \uc$ has a periodic point.
\end{theorem}

This \emph{number} is called the \emph{rotation number} of  the map $F_f$ and denoted by $\rho(F_f)$.  Further, we can show:

\begin{lemma}[\cite{alm00,mis82}]\label{degree:one:result:0} 
	The function $\chi: \mathcal{L}_1' \to \mathbb{R}$ defined by $\chi(F) = \rho(F)$ for all $F \in  \mathcal{L}_1'  $ is continuous.
\end{lemma}

\begin{definition}[\cite{alm00,mis82}]\label{lower:upper:bound:map}
	For $F \in \mathcal{L}_1$,  we define maps $F_l,F_u \in \mathcal{L}_1'$ by ``\emph{pouring water}" from \emph{below} and \emph{above} respectively on the \emph{graph} of $F$   respectively as follows:
	$F_{\ell}(x) = \inf \{ F(y) : y \geqslant x \}$ and $F_u(x) = \sup \{ F(y) : y \leqslant x \} $. We call the \emph{maps}: $F_{\ell}$ and $F_u$,  \emph{lower bound map} and \emph{upper bound map},  corresponding to the \emph{degree one map} $F$ respectively.  
\end{definition}

Let \emph{Const}($G$) be the union of all \emph{open intervals} on which $G$ is \emph{constant}. 

\begin{theorem}[\cite{alm00,mis82}]\label{degree:one:result:1} Then the following statements are true:
	
	\begin{enumerate}
		\item Let $F \in \mathcal{L}_1$. Then, $F_{\ell}(x) \leqslant F(x) \leqslant F_u(x)$ for every $x \in \mathbb{R}$. 
		\item Let $F, G \in \mathcal{L}_1$ with $F \leqslant G$, then $F_{\ell} \leqslant G_{\ell}$ and $F_u \leqslant G_u$. 
		\item Let $F \in \mathcal{L}_1'$. Then, $F_{\ell}= F_u = F$.
		\item The maps $F \mapsto F_{\ell}$ and $F \mapsto F_u$ are Lipschitz continuous with constant $1$ in the sup norm.
		\item If $F_{\ell}(x) \neq F(x)$ for some $x \in \mathbb{R}$, then $x \in Const(F_{\ell})$; similarly for $F_u$,
		
		\item Const($F$) $\subset$ Const($F_{\ell}$) $\cap$ Const($F_u$)
		
	\end{enumerate}

\end{theorem}

\begin{definition}
	Given, a \emph{continuous map},   $f : \uc \to \uc$,   its \emph{lifting} $F_f \in \mathcal{L}_1$ and \emph{point} $y \in \uc$, \emph{periodic} with  \emph{rotation number} $\frac{m}{n}$ and \emph{rotation pair} $(m,n)$,   the set $P' = \pi^{-1}(P) \subset \mathbb{R}$, where  $P \subset \uc$ is the \emph{periodic orbit} of $y \in \uc$ and $\pi: \mathbb{R} \to \uc$ is the \emph{natural projection}, is called   \emph{lifted periodic orbit} of $P$ (equivalently of $y \in \uc$) under $F_f : \mathbb{R} \to \mathbb{R}$. 
\end{definition}

We call $\frac{m}{n}$ and $(m,n)$ respectively  \emph{rotation number} and \emph{rotation pair} of $P \subset \uc$ under $f: \uc \to \uc$ (equivalently of $P'$ under $F_f: \mathbb{R} \to \mathbb{R}$). 

Conversely, \emph{lifted periodic orbits} can also defined independently: a set $Q' \subset \mathbb{R}$ is a \emph{lifted periodic orbit} of a \emph{degree one map} $F$ with \emph{period} $q$ and \emph{rotation number} $\frac{p}{q}$ if for any $ z \in Q'$, $F^q(z) = z+ p$. In such a case, there always exists a \emph{periodic orbit} $Q \subset \uc$ of $f: \uc \to \uc$ \emph{period} $q$ and \emph{rotation number} $\frac{p}{q}$ such that: $Q' = \pi^{-1}(Q)$.

\begin{theorem}[\cite{alm00,mis82}]\label{degree:one:result:2}
	If $F \in \mathcal{L}_1'$ is a lifting of a circle map $f: \uc \to \uc$ and $\rho(F) = \frac{k}{n}$ for coprime integers $k,n$, then $f$ has a cycle $P$ of period $n$ such that the lifted cycle $\pi^{-1}(P)$ of $F$ is disjoint from Const($F$). 
\end{theorem}

\begin{theorem}[\cite{alm00,mis82}]\label{degree:one:result:3} Let $F \in \mathcal{L}_1$. 
	
	 Define:  $F_{\mu} = (min(F, F_{\ell} + \mu))_u$,  for $\mu  \in [0, \lambda]$ and $ \lambda = sup_{x \in \mathbb{R}} (F - F_{\ell})(x)$.

	 Then the following results follow: 
	
	\begin{enumerate}
		\item $F_{\mu} \in \mathcal{L}_1'$ for all $\mu \in [0, \lambda]$,
		
		\item $F_0 = F_{\ell}$ and  $F_{\lambda} = F_u$
		
		\item the map $ \mu \mapsto F_{\mu}$ is Lipschitz continuous with constant 1,
		
		\item the map $ \mu \mapsto \rho(F_{\mu})$ is continuous,
		
		\item if $ \mu \leqslant \lambda $, then $ F_{\mu} \leqslant F_{\lambda}$,
		
		\item each $F_{\mu}$ coincides with $F$ outside $Const(F_{\mu})$,
		
		\item $Const(F) \subset Const(F_{\mu})$ for each $\mu$. 
	\end{enumerate}

\end{theorem}

\begin{definition}[\cite{alm00,mis82}]\label{Upper:lower:rotation:set}
	Let $F \in \mathcal{L}_1$. Then, define \emph{upper-rotation set} of $F$ by  $\overline{Rot(F)} = \{ \overline{\rho_F(x)} : x \in \mathbb{R}\}$; \emph{lower-rotation set} of $F$ by 
$\underline{Rot(F)} = \{ \underline{\rho_F(x)} : x \in \mathbb{R}\}$. Also, define \emph{rotation set} of $F$ by $Rot(F) = \displaystyle \{ \rho_F(x) : x \in \mathbb{R}$ such that $ \overline{\rho_F(x)} = \underline{\rho_F(x)} = \rho_F(x) \}$ are respectively the \emph{lower bound map} and \emph{upper bound map} corresponding to the \emph{degree one map} $F$. 
\end{definition}

\begin{theorem}[\cite{alm00,mis82}]\label{degree:one:result:4}
Let $F \in \mathcal{L}_1$. Then, $\overline{Rot(F)} = \underline{Rot(F)} = Rot(F) = [\rho(F_{\ell}), \rho(F{_u})]$ where $F_{\ell}(x) = \inf \{ F(y) : y \geqslant x \}$ and $F_u(x) = \sup \{ F(y) : y \leqslant x \} $
\end{theorem}

\begin{cor}[\cite{alm00,mis82}]\label{degree:one:result:5}
	The endpoints of $Rot(F)$ depend continuously on $F \in \mathcal{L}_1$.
	
\end{cor}

	\section{Unfolding numbers: Definition and Consistency}\label{first:main:section}
	
 We begin by defining \emph{unfolding numbers} in details. Consider a \emph{continuous map} $f: [0,1] \to [0,1]$ with a \emph{unique point} of \emph{absolute} \emph{maximum} $M_f$ and a \emph{unique point} of \emph{absolute minimum} $m_f$ with $f(M_f) = 1$ and $f(m_f) =0$; assume for definiteness that $M_f <m_f$. Observe that from the standpoint of  \emph{co-existence} of \emph{cycles} of \emph{interval maps},  this is the most important case. Indeed, $m_f >M_f$ implies that there exists two disjoint \emph{open intervals} whose \emph{closures contain their union in their images} (``$1$-\emph{dimensional horseshoe}"), $f$ has the \emph{entropy} at least $\ln2$, \emph{cycles} of all \emph{periods}. So, $f$ is a \emph{chaotic map}. Thus, the case $M_f <m_f$ is the only interesting case (\emph{even though our construction stated below works in the case $M_f>m_f$ too}).

 Consider the \emph{map} $g_f: [0,\frac{1}{2}] \to [0,\frac{1}{2}]$,  which is \emph{conjugate} to the map $ f :[0,1] \to [0,1]$ via the \emph{linear contraction} $\lambda: [0,1] \to [0, \frac{1}{2}]$,  defined by $\lambda(x) = \frac{x}{2} \in [0,\frac{1}{2}]$. We call $g_f$,  the \emph{miniature model map} corresponding to the map $f$.  Let $g_f(0) =c_g$,  $g_f(\frac{1}{2}) = d_g$. Then, the map $g_f$  also has a \emph{unique point} of \emph{absolute} \emph{maximum} $M_g$ and a \emph{unique point} of \emph{absolute minimum} $m_g$ with $f(M_g) = \frac{1}{2}$ and $f(m_g) =0$; also $M_g <m_g$ (See Figure \ref{drawing3}).

 Define $p_f:[0,\frac{1}{2}] \to [\frac{1}{2}, 1]$ by $p_f(x) = 1 - g_f(x)$. Clearly, $p_f(0) = 1-c_g$, $p_f(M_g) = \frac{1}{2}$, $ p_f(m_g) = 1$, $p_f(\frac{1}{2}) = 1- d_g$.  The \emph{graph} of $p_f(x)$ is obtained by \emph{reflecting vertically}, the \emph{graph} of $g_f(x)$ with the respect with the line $y = \frac{1}{2}$.

 Define $q_f:[\frac{1}{2}, 1] \to [\frac{1}{2}, 1]$ by $q_f(x) =  p_f(1-x) =1 - f(1-x)$. Clearly, $q_f(\frac{1}{2}) = 1-d_g$, $q_f(1-m_g) = 1$, $ q_f(1-M_g) = \frac{1}{2}$, $q_f(1) = 1- c_g$. The \emph{graph} of $q_f(x)$ is obtained  by \emph{reflecting horizontally}, the \emph{graph} of  $p_f(x)$ across the line $ x = \frac{1}{2}$.

 Define $r_f:[\frac{1}{2}, 1] \to [1, \frac{3}{2}]$ by $r_f(x) = 2 - q_f(x) = 1 + f(1-x)$. Clearly, $r_f(\frac{1}{2}) = 1+d_g$, $r_f(1-m_g) = 1$, $ r_f(1-M_g) = \frac{3}{2}$, $r_f(1) = 1+c_g$. 
 Thus, the \emph{graph} of  $r_f$ is obtained from the \emph{graph} of  $q_f(x)$ by \emph{reflecting vertically} across the line $y = 1$.

 We now \emph{construct} a map $F_f : [0,1] \to [1, \frac{3}{2}]$ as follows (See Figure \ref{drawing3}):

 $ F_f(x) =\begin{cases} 
 	
 	g_f(x)  &\text{if $ 0 \leqslant x \leqslant M_g $}\\ 
 	p_f(x)   &\text{if $M_g \leqslant x \leqslant \frac{1}{2}$.} \\
 	q_f(x)  &\text{if $\frac{1}{2} \leqslant x \leqslant 1-m_g $.} \\
 	r_f(x)  &\text{if $1-m_g \leqslant x \leqslant 1$.} \\

 \end{cases} $

Observe that $F_f(1) = F_f(0) + 1$.  \emph{Extend} $F_f: \mathbb{R} \to \mathbb{R}$ by  $F_f(x+1) = F_f(x) +1 $ for all $x \in \mathbb{R}$. 	We shall call the map $F_f$,  the \emph{heaved map} corresponding to the map $f$.  By construction $F_f(x)$ is a \emph{degree one map} of the real line (See Figure \ref{drawing3}).

\begin{figure}[H]
	\caption{Graph of $F_f$ shown in \emph{dotted line}}
	\centering
	\includegraphics[width=0.7 \textwidth]{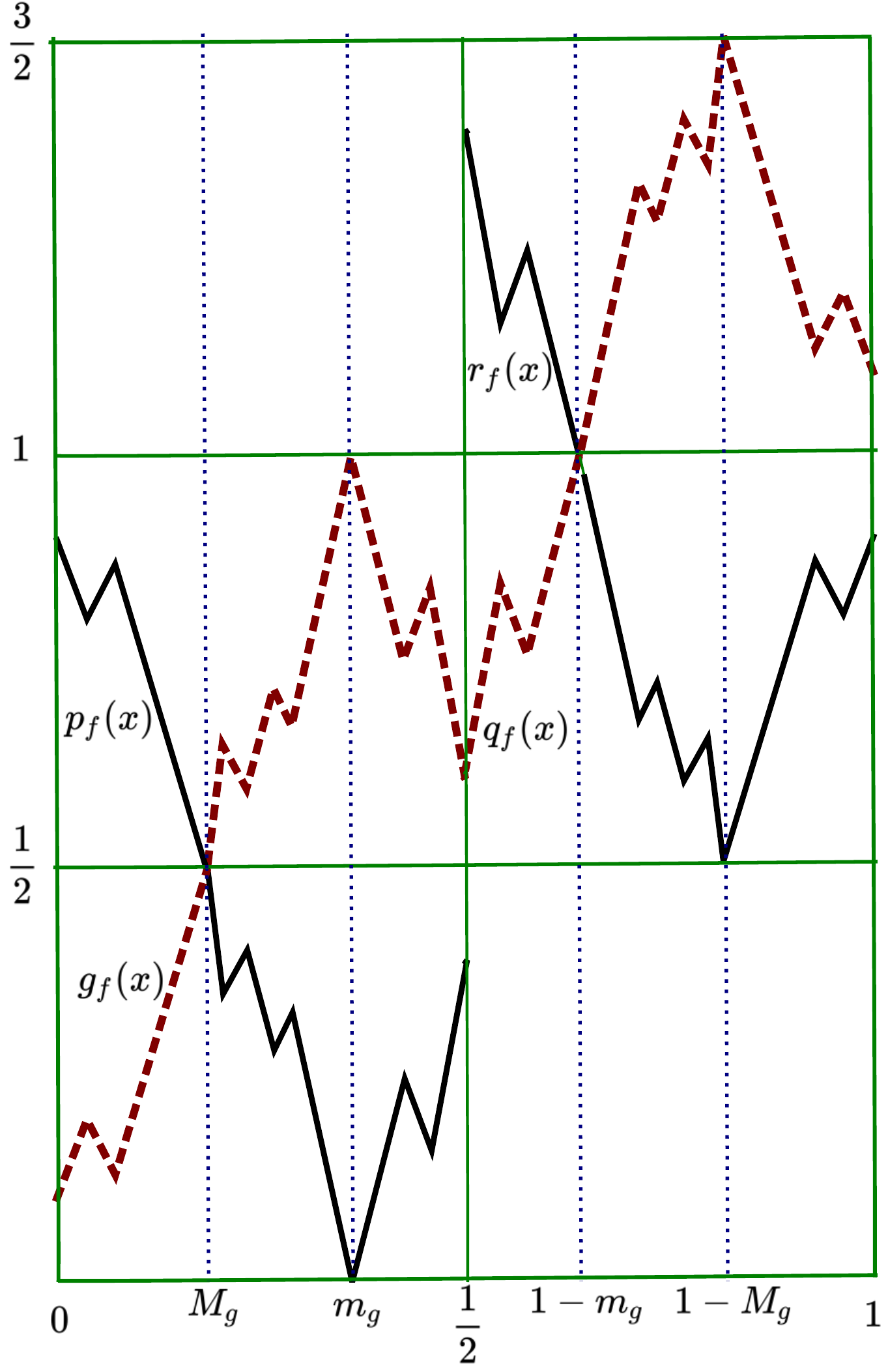}
	\label{drawing3}
\end{figure}

\begin{definition}\label{unfolding:number}
	For $x \in [0,1]$,  the \emph{upper unfolding number} of $x$ under $f: [0,1] \to [0,1]$, denoted by $\overline{u_f(x)}$, is defined as the \emph{upper rotation number} of $\lambda (x) = \frac{x}{2} \in [0,\frac{1}{2}]$ under $F_f: \mathbb{R} \to \mathbb{R}$,  that is, $\overline{u_f(x)}  = \overline{\rho_{F_f}(\lambda(x))}  = \displaystyle \lim_{n \to \infty} \mathrm{sup}  \displaystyle  \frac{ F^{n}(\lambda (x)) - \lambda (x)}{n}$.

	Similarly, the \emph{lower unfolding number} of $x \in [0,1]$ under $f: [0,1] \to [0,1]$, denoted by $\underline{u_f(x)}$, is defined as the \emph{lower rotation number} of $
	\lambda(x) = \frac{x}{2}  \in [0,\frac{1}{2}]$ under $F_f: \mathbb{R} \to \mathbb{R}$, that is, $\underline{u_f(x)}  = \underline{\rho_{F_f}(\lambda(x))} = \displaystyle \lim_{n \to \infty}  \mathrm{inf} \displaystyle  \frac{ F^{n}(\lambda (x)) - \lambda (x)}{n}$. 
	
	\smallskip
	
	Clearly, $\underline{u_f(x)}  \leqslant \overline{u_f(x)}$ for all $x \in [0,1]$. The \emph{set} $\mathcal{U}_f(x) = \{ \underline{u_f(x)}  , \overline{u_f(x)} \}$ is called the \emph{unfolding set} of  $x$ \emph{under} $f: [0,1] \to [0,1]$.  If $\underline{u_f(x)}  = \overline{u_f(x)} = u_f(x)$ for some $x \in [0,1]$, we call: $u_f(x)$, the \emph{unfolding number} of $x \in [0,1]$ \emph{under} $f: [0,1] \to [0,1]$.

\end{definition}

\begin{definition}\label{Upper:lower:unfolding:set}
	Let $f: [0,1] \to [0,1]$. Then, define \emph{upper-unfolding set} of $f$ by  $\overline{Unf(f)} = \{ \overline{u_f(x)} : x \in [0,1]\}$; \emph{lower-unfolding set} of $f$ by 
	$\underline{Unf(f)} = \{ \underline{u_f(x)} : x \in [0,1]\}$. Also, define \emph{unfolding set} of $f$ by $Unf(f) = \{ u_f(x) : x \in \mathbb{R}$ such that $ \overline{u_f(x)} = \underline{u_f(x)}= u_f(x)  \}$. 
\end{definition}

We now establish the \emph{consistency} of the theory of \emph{unfolding numbers}. We begin with a simple \emph{lemma}:

\begin{lemma}\label{pull:back}
	Given any $ z \in \mathbb{R}$, there exists $ x \in [0, \frac{1}{2}]$ such that $F_f^{n}(x) = z$ for some $n \in \mathbb{N}$. 
\end{lemma}

\begin{proof}
	Construction of the map $F_f$ reveals that for any $ \ell \in \left [ \displaystyle n\frac{\pi}{2},  (n+1)\frac{\pi}{2}\right ]$,  $n \in \mathbb{Z}$, there exists $\ell' \in \left [ \displaystyle (n-1)\frac{\pi}{2},  n\frac{\pi}{2}\right ]$ such that $F_f(\ell') = \ell$. Hence, the result follows. 
\end{proof}

\begin{definition}
	Let $F \in \mathcal{L}_1$ and $y,z \in \mathbb{R}$. The \emph{trajectories} $T_F(y)$ and $T_F(z)$ are said to be \emph{coincident},  if $F^i(y) = F^i(z)$ for all $i\geqslant 0$, and \emph{eventually coincident} if $F^i(y) = F^i(z)$ for all $i \geqslant n_0$ for some $n_0 \in \mathbb{N}$.
	
\end{definition}

The following \emph{result} follows easily and its \emph{proof} is left to the reader: 

\begin{lemma}\label{eventually:coincident}
	Let $F \in \mathcal{L}_1$ and $y,z \in \mathbb{R}$.  Suppose the trajectories $T_F(y)$ and $T_F(z)$ are  coincident or eventually coincident, then the following holds:
	
	\begin{enumerate}
		\item $\overline{\rho_{F}(y)} =  \overline{\rho_{F}(z)} $ and $\underline{\rho_{F}(y)} =  \underline{\rho_{F}(z)} $. 
		
		\item if  the trajectories  $T_F(y)$ and $T_F(z)$  are lifted periodic orbits under $F$, they have the same period, same rotation pair and same modified rotation pair. 
	\end{enumerate}
	
\end{lemma}

\begin{theorem}\label{periodic:thm1}
	If $P$ is a periodic orbit of $f: [0,1] \to [0,1]$ of period $q$, then the unfolding number of each point of $P$ is same and is of the form $\frac{p}{q}$ for $p,q \in \mathbb{N}$. 
\end{theorem}

\begin{proof}

	Let $x \in P$. Then, $\lambda(x) \in [0, \frac{1}{2}]$ is \emph{periodic} under $g_f : [0,\frac{1}{2}] \to [0,\frac{1}{2}] $ with same \emph{period} $q$. Construction of the map $F_f: \mathbb{R} \to \mathbb{R}$,  then implies  that, the \emph{trajectory} $T_{F_f}(\lambda(x))$  of $\lambda(x)$ under $F_f$ is a \emph{lifted periodic orbit} of $F_f$ of the same \emph{period} $q$. Suppose the \emph{rotation number} of $\lambda(x)$ under $F_f$ is $\frac{p}{q}$, then,   definition \ref{unfolding:number},  implies that, the \emph{unfolding number} of $x$ under $f$ is $\frac{p}{q}$, and hence the result follows. 
\end{proof}

\begin{definition}\label{modified:unfolding:pair}
	We call $\frac{p}{q}$,  $p,q \in \mathbb{N}$ (from Theorem \ref{periodic:thm1}),  the \emph{unfolding number} of the \emph{periodic orbit} $P=T_f(x)$ of $f$ and the \emph{ordered pair} $(p, q)$, the \emph{unfolding pair} of $P$; we write:  $un(P) = \frac{p}{q}$ and $up(P) = (p,q)$.  If $ \frac{p}{q} = \frac{r}{s}$,  where $r$ and $s$ are \emph{coprime} and $ \frac{p}{r} = \frac{q}{s} = m \in \mathbb{N} $,  then we call the \emph{ordered pair} $(\frac{r}{s}, m)$, the \emph{modified unfolding pair} of $P$ (\emph{similar to rotation theory for degree one circle maps, see Section} \ref{introduction}); we write $mup(P) = (\frac{r}{s}, m)$. Let $mup(f)$ be the set of all \emph{modified unfolding pairs} of \emph{cycles} of $f$. 
	
\end{definition}

\begin{theorem}\label{conjugacy:thm:3}
	The	set of modified unfolding pairs of cycles of the interval map $f: [0,1] \to [0,1]$ is equal to the set of modified rotation pairs of lifted cycles of the degree one circle map $F_f: \mathbb{R} \to \mathbb{R}$, that is, $mup(f) = mrp(F_f)$. 
	
\end{theorem}

\begin{proof}
	Let $P$ be a \emph{cycle} of $f: [0,1] \to [0,1]$ with \emph{period} $mq$ with \emph{unfolding number} $\frac{p}{q}$, with $p,q$ \emph{co-prime} and \emph{modified unfolding pair} $(\frac{p}{q}, m)$ for some $m \in \mathbb{N}$. Let $x \in P$. Then, $\lambda(x) \in [0, \frac{1}{2}]$ is a \emph{periodic point} of $g_f : [0,\frac{1}{2}] \to [0,\frac{1}{2}] $ of same \emph{period} $mq$. Construction of the map $F_f: \mathbb{R} \to \mathbb{R}$,  then stipulates that the \emph{trajectory} $T_{F_f}(\lambda(x))$  of $\lambda(x)$ under $F_f$ is a \emph{lifted periodic orbit} of $F_f$.  Moreover, definitions \ref{unfolding:number} and  \ref{modified:unfolding:pair}  entails that  \emph{modified rotation pair} of $T_{F_f}(\lambda(x))$   is $(\frac{p}{q}, m)$.  Hence, $(\frac{p}{q}, m) \in mrp(F_f)$. 
	
	Conversely, suppose $(\frac{p}{q}, m)$ be the \emph{modified rotation pair} of a \emph{lifted cycle} $P'$ of $F_f$. Choose $z \in P'$. By Lemma \ref{pull:back}, there exists $ x \in [0, \frac{1}{2}]$ and $n \in \mathbb{N}$ such that $F_f^{n}(x) = z$. Since, the \emph{trajectories} $T_{F_f}(z)$  and  $T_{F_f}(x)$ are \emph{eventually coincident}, by Lemma \ref{eventually:coincident}, the \emph{modified rotation pair} of $T_{F_f}(x)$ is also $(\frac{p}{q}, m)$.  Connection between the \emph{maps} $g_f: [0, \frac{1}{2}] \to [0, \frac{1}{2}]$ and $F_f: \mathbb{R} \to \mathbb{R}$, warrants that $x\in [0, \frac{1}{2}]$ is a  \emph{periodic point} of $g_f$ of \emph{period} $mq$ and hence,  $\lambda^{-1}(x) = y \in [0,1]$ is a \emph{periodic point} of $f$ of the same \emph{period} $mq$. Definitions \ref{unfolding:number} and \ref{modified:unfolding:pair},  now yields that, the  \emph{modified unfolding pair} of  $P= T_f(y)$  is  equal to $(\frac{p}{q}, m)$, and hence $(\frac{p}{q}, m) \in mup(f)$ establishing the reverse inclusion.
\end{proof}

\begin{theorem}\label{completeness:1}
	There exists $t_1, t_2 \in \mathbb{R}$ and $m_1, m_2 \in \mathbb{N}$ such that $mup(f)= [(t_1, m_1), (t_2, m_2)]$. 
\end{theorem}

\begin{proof}
	
	By Theorem \ref{circle:maps}, there exists \emph{modified rotation pairs} $(t_1, m_1)$ and $(t_2, m_2)$ where $t_1, t_2 \in \mathbb{R}$ and $m_1, m_2 \in \mathbb{N}$ such that the  set of all \emph{modified rotation pairs} of \emph{cycles} of the \emph{degree one map} $F_f : \mathbb{R} \to \mathbb{R}$ is given by $mrp(F_f) = [(t_1, m_1), (t_2, m_2)]$. Now, the result  follows from Theorem \ref{conjugacy:thm:3}.
\end{proof}

Theorems \ref{periodic:thm1} and \ref{completeness:1} guarantees the \emph{completeness} of the \emph{unfolding numbers}.

\section{Unfolding Interval and its realization}\label{unfolding:interval:section}

In this section, we analyze the \emph{structure} of the \emph{set} of \emph{unfolding numbers} of \emph{points} under a given \emph{interval map} $f:[0,1] \to [0,1]$.

\begin{theorem}\label{unfolding:set}\label{existence:unfolding:interval}
	Let $f: [0,1] \to [0,1]$ be a given continuous interval map. Then, 
	
$\overline{Unf(f)} = \underline{Unf(f)} = Unf(f) =[u_f, \frac{1}{2}]$, where $u_f = \rho((F_f)_{\ell})$ where ${(F_f)}_{\ell}$ represents the lower bound map corresponding to the degree one map $F_f$.

\end{theorem}

\begin{proof}
	We begin by showing: $\underline{Unf(f)} = \underline{Rot(F_f)}$. Choose $ x \in [0,1]$ and $\underline{u_f(x)} \in \underline{Unf(f)}$. Then, $\lambda(x) \in [0, \frac{1}{2}]$ and hence,  by definition \ref{unfolding:number}, $\underline{u_f(x)} = \underline{\rho_{F_f}(\lambda(x))} \in \underline{Rot(F_f)}$,  and hence $\underline{Unf(f)} \subseteq \underline{Rot(F_f)}$.  Now,  take $ y \in \mathbb{R}$ and $\underline{\rho_{F_f}(y)} \in  \underline{Rot(F_f)}$. By Lemma \ref{pull:back}, there exists  a point $ z \in [0, \frac{1}{2}]$ such that $F_f^{n}(z) = y$ for some $ n \in \mathbb{N}$. Then, the \emph{trajectories} $T_{F_f}(z)$ and $T_{F_f}(y)$ are \emph{eventually coinciding} and hence, by Lemma \ref{eventually:coincident}:  $\underline{\rho_{F_f}(y)} =  \underline{\rho_{F_f}(z)} $.  Let $\lambda^{-1}(z) = x \in [0,1]$. Then, by Definition \ref{unfolding:number},   $ \underline{\rho_{F_f}(y)} = \underline{\rho_{F_f}(z)} = \underline{u_f(x)} \in \underline{Unf(f)} $  implying the reverse inclusion.

Similarly, it can be shown that: $\overline{Unf(f)} = \overline{Rot(F_f)}$. Now,  Theorem  \ref{degree:one:result:4} and Definition \ref{Upper:lower:unfolding:set},  yields that: $\overline{Unf(f)} = \underline{Unf(f)} = Unf(f) = [\rho({(F_f)}_{\ell}), \rho({(F_f)}_u)]$, where ${(F_f)}_{\ell}$ and $(F_f)_u$ are respectively the \emph{lower bound map} and \emph{upper bound map} corresponding to the \emph{degree one map} $F_f$.  It remains to show $\rho((F_f)_{u}) = \frac{1}{2}$.

It is easy to see that $(F_f)_u(x) = \frac{1}{2}$ in $[0,M_g]$, $(F_f)_u (x) = F_f (x) $ in $[M_g,m_g]$ and $(F_f)_u (x) = 1 $ in the interval $ [m_g, 1-m_g]$. It follows that $(F_f)_u(0) = \frac{1}{2}$ and $(F_f)_u(\frac{1}{2})  = 1$ and hence the \emph{trajectory}: $ T_{(F_f)_u}(0)$ of $0$ under the map $(F_f)_u $,   has \emph{rotation number} $\frac{1}{2}$.  Hence, the result follows.  
\end{proof}

\begin{definition}
	The \emph{interval} $Unf(f) = [u_f, \frac{1}{2}]$ from Theorem \ref{existence:unfolding:interval} is called the \emph{unfolding interval} of the interval \emph{map} $f: [0,1] \to [0,1]$. 
\end{definition}

\begin{definition}
	Given a map $f: [0,1] \to [0,1]$, a \emph{set} $A \subset [0,1]$ and a \emph{point} $ y \in [0,1]$, we define:
	
	\begin{enumerate}
		\item $ \displaystyle \eta_A(y) = \mathrm{sup} \{ x \in A : f(x) = f(y)\}$ 
		\item $\xi_A(y) = \mathrm{inf} \{ x \in A : f(x) = f(y)\}$
		
		\item $Y_f =  \eta_{[0,M_f]}([0,M_f]) $ $ \cup \eta_{[0,1]}([M_f,m_f]) $ $ \cup \xi_{[0,1]}([M_f,m_f]) \cup $ $  \eta_{[m_f,1]}([m_f,1]) $. 
	\end{enumerate}

\end{definition}

Let us denote by $\mathcal{C}$, the set of \emph{open subsets} of $\mathbb{R}$ where the \emph{lower bound function}:    $(F_f)_{\ell} (x) $ is \emph{constant} and let $\mathcal{U}$ to be the \emph{union} of \emph{open intervals} from $\mathcal{C}$.

\begin{theorem}\label{unfolding:number:measure:prelim}
	Let $f : [0,1] \to [0, 1]$ be a continuous map with unfolding interval $[u_f, \frac{1}{2}]$.
	Then for any point $ y \in [0,1]$ whose trajectory is contained in $Y_f$, the unfolding set,  $\mathcal{U}_f(y) = \{ u_f \}$. 
	
\end{theorem}

\begin{proof}
	Construction of the maps $g_f$ and $F_f$ reveals that for $y \in [0,1]$,  if the \emph{trajectory} $T_{f}(y)$ of $y$ under $f$ is \emph{disjoint} from $Y_f$, then the \emph{trajectory} $T_{F_f}(\lambda(y))$ of $\lambda(y)$ under $F_f$ is \emph{disjoint} from $\mathcal{U}$.  From Theorem \ref{degree:one:result:1}(5),  this means that, the \emph{trajectory} $T_{F_f}(\lambda(y))$ of $\lambda(y)$ under $F_f$ and the  \emph{trajectory} $T_{(F_f)_{\ell}}(\lambda(y))$ of $\lambda(y)$ under  the \emph{lower bound map} $(F_f)_{\ell}$ are  are \emph{coincident}. So, Theorem \ref{existence:unfolding:interval} and Lemma \ref{eventually:coincident},  entails that \emph{rotation number} of $\lambda(y)$ under $F_f$ \emph{exists} and equals $u_f$:   the \emph{rotation number} of $\lambda(y)$ under the \emph{lower bound map} $(F_f)_{\ell}$. Finally, definition \ref{unfolding:number}, furnishes that the \emph{unfolding set} $\mathcal{U}_f(y)$  of $y$ is precisely equal to $\{ u_f \}$.
\end{proof}

 By a \emph{minimal} set of a \emph{map} $h$
of a \emph{compact space} to \emph{itself}, we mean an \emph{invariant compact set} $Z$ all
of whose points have \emph{dense trajectories} in $Z$ (thus, a \emph{compact
invariant subset} of $Z$ \emph{coincides} with $Z$).  Let us call a \emph{maximal} (\emph{in terms of set inclusion}) \emph{open set}  where a \emph{continuous function} $h$  is  \emph{constant}, a \emph{flat spot} of  $h$. The following \emph{result} follows: 

\begin{theorem}\label{unfolding:number:measure:1}
	
	Let $f : [0,1] \to [0, 1]$ be a continuous map with unfolding interval $[u_f, \frac{1}{2}]$. Then, there exists a minimal $f$-invariant set $Z_f \subset  Y_f$ such that for any point $ y \in Z_f$, we have $\mathcal{U}_f(y) = \{ u_f \}$ and there are two possibilities :
	
	\begin{enumerate}
		
		\item $u_f$ is rational, $Z_f$ is a periodic orbit, and
		$f|_{Z_f}$ is canonically conjugate to the circle rotation by $u_f$
		restricted on one of its cycles. 
		
		\item $u_f$ is irrational, $Z_f$ is a Cantor set, and $f|_{Z_f}$
		is canonically  semi-conjugate to the circle
		rotation by $u_f$.
		
	\end{enumerate}

\end{theorem}

\begin{proof}
	The \emph{natural projection} map $\pi: \mathbb{R} \to \uc$ \emph{monotonically semi-conjugates} $(F_f)_{\ell}$ to a \emph{degree one map} $\tau_f: \uc \to \uc$. Further, since  $(F_f)_{\ell}$ is \emph{monotonically non-decreasing},  $\tau_f :\uc \to \uc$ \emph{preserves}  \emph{cyclic orientation} in the 
	\emph{non-strict sense}. From Theorem \ref{degree:one:result:2}, it follows that there exists a \emph{point} $ z \in \mathbb{R}$, such that, the \emph{trajectory} $T_{F_f}(z)$ of $z$ under $F_f$ is \emph{disjoint} from $\mathcal{U}$, (that is, the \emph{trajectories} $T_{F_f}(z)$ and $T_{(F_f)_{\ell}}(z)$ are \emph{coincident}), such that $\rho_{F_f}(z) = \rho_{(F_f)_{\ell}}(z) =  \rho((F_f)_{\ell}) = u_f$. 
	We have two cases here:
	
	\begin{enumerate}
		\item 
		If $u_f= \frac{p}{q}$ is \emph{rational}, then $\pi(z) \in \uc$ is \emph{periodic} under $\tau_f$ of \emph{period} $q$. If $P_f$ be the \emph{orbit} of $\pi(z)$, under $\tau_f$, then the map $\tau$ \emph{restricted} on $P_f$ acts as \emph{circle rotation} by \emph{angle} $\frac{p}{q}$. 
		
		\item  If $u_f$ is \emph{irrational}, then $\pi(z) \in \uc$  belongs to the \emph{omega limit set}  $\omega_{\tau_f}\left (\pi(z) \right ) = Q_f \subset \uc$ of $\pi(z)$ under $\tau_f$. It is \emph{minimal} (that is, $\tau_f$-\emph{orbits} of all points of $Q_f$ are \emph{dense} in $Q_f$)  such that \emph{collapsing arcs} of $\uc$ complementary to $Q_f$, we can \emph{semi-conjugate} $\tau_f : \uc \to \uc$ to \emph{irrational rotation} of $\uc$ by the \emph{angle} $u_f$. In this case, clearly, $Q_f$ is a \emph{cantor set}.

	\end{enumerate}
	We need to now find the appropriate $f$-\emph{invariant set} $Z_f \subset [0, 1]$ whose existence is claimed in the theorem. The \emph{correspondence} between the maps involved in our construction, Lemma \ref{pull:back},  implies the existence of a \emph{continuous conjugacy} $\psi_f : \uc \to [0,1 ]$ between the maps $\tau : \uc \to \uc$ and $f : [0, 1] \to [0, 1]$ applicable outside $\pi(\overline{\mathcal{U}}) \subset \uc$.

		If $u_f$ is \emph{irrational}, then clearly, we can easily find a point $y \in Q_f \subset \uc$ such that $T_{\tau_f}(y) \cap \pi (\overline{\mathcal{U}}) = \phi $. Then, it follows that there is a \emph{continuous conjugacy} between $\tau|_{T_{\tau_f}(y)}$ and $f|_{Z_f}$ where $Z_f$ is the \emph{omega limit set} $\omega_f(\psi_f(y))$ of $\psi_f(y)$ under $f$. It is then easy to see that $Z_f$ is a \emph{Cantor Set} and has all the desired properties. 
	
	Suppose $u_f = \frac{p}{q}$ is \emph{rational}.  Since, by definition, $T_{F_f}(z) \cap \mathcal{U} = \phi$, it follows that,  $P_f \cap \pi(\mathcal{U}) = \phi $, but, it is not guaranteed that $P_f \cap \pi (\overline{\mathcal{U}}) = \phi$. However, as we now show,  the specifics of the \emph{construction} helps us to elude this problem. If $P_f \cap \pi (\overline{\mathcal{U}}) = \phi $, then we are done by the arguments similar to the ones from the previous paragraph. Suppose that $P_f$ passes through an \emph{endpoint} $\eta$ of a \emph{flat spot} of $\tau_f$. Choose a \emph{point} $y$ very \emph{close} to $\eta$ avoiding \emph{flat spots} of $\tau_f$. Then, the \emph{finite segment} of the $\tau_f$ orbit of $y$ consisting of \emph{points} $ y, \tau_f(y), \dots \tau_f^{q}(y) \approx y$ is \emph{transformed} by $\psi_f$ into a \emph{finite segment} $\psi_f(y), f(\psi_f(y)), \dots f^q(\psi_f(y)) \approx \psi_f(y)$ which \emph{converges} to an $f$-\emph{periodic orbit} as $y \to \eta$. The \emph{limit periodic orbit} $Z_f$ consisting of $\psi_f(y), f(\psi_f(y)), \dots f^q(\psi_f(y)) \approx \psi_f(y)$ has all the desired properties. 
	
\end{proof}

Theorem \ref{unfolding:number:measure:1} leads us to the following result:

\begin{theorem}\label{unfolding:number:measure:2}
		Let $f : [0,1] \to [0, 1]$ be a continuous map with unfolding interval $[u_f, \frac{1}{2}]$. Then, there exists a point $z \in [0, 1]$ such that the map $f|_{Z_f}$ where $Z_f = \omega_f(z)$ is strictly ergodic. 
\end{theorem}

\begin{proof}
	It follows from Theorem \ref{unfolding:number:measure:1}, that there exists a \emph{point} $z \in [0, 1]$ such that the map $f|_{Z_f}$ where $Z_f = \omega_f(z)$ has a unique \emph{invariant measure} $\mu_f $ such that the \emph{measure} $\mu_f$ can be \emph{transformed} in  a \emph{canonical fashion} to a specific \emph{invariant measure} related to \emph{circle rotation} by angle $u_f$. If $u_f$ is \emph{rational}, the \emph{measure} $\mu_f$ is simply the CO-\emph{measure} concentrated on the $f$-\emph{periodic orbit} of $z$ (see Section \ref{introduction}) whereas if $u_f$ is \emph{irrational}, $\mu_f$ corresponds in a canonical fashion to the \emph{Lebesgue measure} on $\uc$ \emph{invariant} under \emph{irrational rotation} on angle $u_f$. Hence, the result follows.  
	\end{proof}

	\section{Characterizing dynamics of a cycle with a given unfolding number }\label{reading:unfolding:comb}
	
As demonstrated in Section \ref{first:main:section}, computing \emph{unfolding number} of a \emph{periodic orbit} $P$ typically requires first constructing, the \emph{heaved map} $F_f$, where $f$ is the $P$-\emph{linear map} \emph{corresponding} to the \emph{periodic orbit} $P$.  In this section, we reveal that \emph{unfolding number} of  a \emph{periodic orbit} $P$ can be read directly from the \emph{combinatorics} of the \emph{permutation} \emph{corresponding} to $P$ eliminating need to construct $F_f$.

We commence by introducing a few notions which would be used throughout this section:  Let $P$ be a \emph{periodic orbit} of \emph{period} $q$.  Let $f:[0,1] \to [0,1]$ be a $P$-\emph{linear} map such that the \emph{left-end point} and \emph{right-end point} of $P$ are $le(P) =0$ and $ri(P)=1$ respectively. Let $m_f$ and $M_f$ be points of \emph{absolute minimum} and \emph{absolute maximum} of $f$ such that $f(m_f) =0$ and $f(M_f) = 1$. Likewise, $m_g$ and $M_g$ are points of \emph{absolute minimum} and \emph{absolute maximum} of the \emph{miniature model map} $g_f$, corresponding to $f$,  such that $f(m_g) =0$ and $f(M_g) = \frac{1}{2}$.

 \begin{definition}\label{unfolding:set:defn}

 	 By \emph{unfolding index set} of $P$, we mean the \emph{maximal} (\emph{in terms of set inclusion}) set $\mathcal{N}(P) = \{ n_1,$ $ n_2, \dots n_k\}$  of \emph{non-negative integers} satisfying the following properties: 
 	
 	\begin{enumerate}
 		\item $0  \leqslant n_i <  n_{i+1} < q$, for all $i \in \{0, 1,2,\dots k-1\}$. 
 		
 		\item $n_1$ is the \emph{smallest non-negative integer} such that $f^{n_1-1}(0) < m_f$ and $f^{n_1 -2}(0) >M_f$.

 		\item  $n_{i+1}$ is the \emph{smallest non-negative integer}, \emph{greater than or equal to}  $n_i +2$ but strictly \emph{less} than $q$ such that:  $f^{n_{i+1}-1}(0) < m_f$ and $f^{n_{i+1} - 2}(0) >M_f$,  for all $i \in \{0, 1,2,\dots k-1\}$. 
 	\end{enumerate}
 	
 	The \emph{points}  $f^{n_i}(0) \in P$, $n_i \in \mathcal{N}(P) $  are called  \emph{unfolding points} of $P$. 
 
 \end{definition}

\begin{theorem}\label{unfolding:computation}
If  $P$ has $p$  unfolding points, then, the unfolding number of $P$ is $\frac{p}{q}$. 
\end{theorem}

\begin{proof}

\emph{Definition} of $\mathcal{N}(P)$ coupled with the \emph{construction} of the \emph{maps} $g_f$ and $F_f$ reveals that the \emph{trajectory} of $0$ under $g_f$  and \emph{trajectory} of $0$ under $F_f$ are the \emph{same} till the first $n_1 -2$ \emph{iterates}, that is, $g_f^{i}(0) = F_f^{i}(0)$, for $i \in \{0,1,2, \dots , n_1-2\}$. Furthermore,  since $f^{n_1 -2}(0) > M_f$, we have: $g_f^{n_1-2}(0) = F_f^{n_1-2}(0) > M_g$.  Consequently, $F_{f}^{n_1 -1}(0) > \frac{1}{2}$. Moreover, since, $f^{n_1-1}(0) < m_f $, we must have: $g_f^{n_1-1}(0) < m_g$. So,  $1-m_g < F_f^{n_1-1}(0) < 1 $ and hence, $F_{f}^{n_1}(0) \geqslant 1$.

The above mentioned scheme persists with $F_f^{i}(0) \in \left (1,2 \right )$, for $ n_1 < i  \leqslant  n_2 -1$ and $F_f^{n_2}(0) \geqslant 2$; then,  $F_f^{i}(0) \in (2,3)$ for $ n_2 < i \leqslant n_3 -1$ and $F_f^{n_3}(0)  \geqslant 3$, and so forth, leading to $F_f^{n_p}(0) \geqslant p $.   We assert that $F_f^{q}(0) = 0 + p$. Indeed,  since $T_f(0) =P$ is a \emph{periodic orbit} of $f$ of period $q$, $T_{F_f}(0)$ is a \emph{lifted periodic orbit} of $0$ of  same \emph{period} $q$. Thus,  $F_f^{q}(0) = 0 + \ell $ for some $\ell \in \mathbb{N}$. Since, $n_p \leqslant q$, it follows that,  $\ell \geqslant p$. If $ \ell \geqslant p+1$, then there exists $m \in \mathbb{N} , m > n_p$ such that $F_f^{m}(0) > 1$. However, this is feasible only  if, there exists $m \in \mathbb{N},  m > n_p$ such that $f^{m-1}(0) < m_f$ and $f^{m-2}(0) > M_f$, contradicting the \emph{maximality} of  $\mathcal{N}(P)$. Thus, $\ell = p$, establishing the result. 
\end{proof}

Theorem \ref{unfolding:computation} enables us to define \emph{unfolding numbers} in a manner reminiscent of the traditional approach used for defining \emph{rotation numbers}, as expounded in Section \ref{first:main:section}. We define an \emph{observable} $\Phi: P \to \mathbb{R}$ as follows. For each $i \in \mathbb{N}$ define: 

\begin{equation}
	\nonumber
	\Phi(f^i(0))=
	\begin{cases}
		
		1 & \text{if $i \in \mathcal{N}$}\\
		0 &   \text{if $i \notin \mathcal{N}$}\\
			\end{cases}
\end{equation}
Then the following easily follows from Theorem \ref{unfolding:computation}.

\begin{cor}\label{unfolding:traditional}
The unfolding number of $P$ is $\frac{1}{q}  \sum_{i=1}^{q} \Phi(f^i(0))$ and the unfolding pair of $P$ is: $( \sum_{i=1}^{q} \Phi(f^i(0)), q)$.
\end{cor}

\section{Forced Unfolding Interval}\label{forced:unfolding:interval}

	In this section we introduce the idea of \emph{forced unfolding interval}. We define:  \emph{unfolding number}, \emph{unfolding pair} and \emph{modified unfolding pair} of a \emph{pattern} $\pi$ to be the \emph{unfolding pair}, \emph{unfolding pair} and \emph{modified unfolding pair} of any \emph{cycle} $P$ which exhibits $\pi$. The \emph{unfolding number}, \emph{unfolding pair} and \emph{modified unfolding pair} of a \emph{pattern} $\pi$ are respectively denoted by $un(\pi), up(\pi)$ and $mup(\pi)$. We now introduce the notion of \emph{forced unfolding interval}: 
 \begin{definition}
 	We say that a \emph{pattern} $\pi$ \emph{forces} the \emph{unfolding interval} $U_{\pi}$,  if $U_{\pi} = [u_{\pi}, \frac{1}{2}]$ is the \emph{unfolding interval} of the $P$-\emph{linear map} $f$ where $P$ \emph{exhibits} $\pi$.
 \end{definition}

 It follows that: if $\rho(\pi)$ is the \emph{unfolding number} of $\pi$, then,  $ u_{\pi} \leqslant \rho({\pi})$. 
\begin{definition}\label{divergent}
	We call a cycle $P$ \emph{divergent} (or the \emph{pattern $\pi$ it exhibits}) if there are points $x,y,z$ of $P$ such that:
	
	\begin{enumerate}
		\item $x < y < z$
		
		\item $f(x) < x$, $f(y) \geqslant z$ and $f(z) \leqslant x$. 
	\end{enumerate}
	where $f$ is the $P$-\emph{linear map}. We call a \emph{pattern} $\pi$ \emph{divergent},  if any \emph{cycle} $P$ which \emph{exhibits} it is \emph{divergent}. 
\end{definition}

It is well known that if a  \emph{pattern} $\pi$ is \emph{divergent} then for any cycle $P$ of this \emph{pattern}, the $P$-\emph{linear map} $f$ has a \emph{horseshoe} and has \emph{cycles} of all possible \emph{periods}. We show that in fact,  in such a case, $f$ has \emph{cycles} of all possible \emph{unfolding pairs}.

	A \emph{cycle} $P$ (or the \emph{pattern} $\pi$ it \emph{exhibits}) is called \emph{convergent},  if it is not \emph{divergent}. There is another equivalent way to define \emph{convergent pattern}
	 Namely, let $\mathcal{U}$ be the \emph{family} of all \emph{interval maps} with a \emph{unique fixed point}. If $f$ is a $P$-\emph{linear map} for a cycle $P$, then $P$ is \emph{convergent} if and only if $f \in \mathcal{U}$.

\begin{theorem}\label{divergent:unfolding:interval}
	If $f$ has a divergent cycle, then the unfolding interval of $f$ is $[0, \frac{1}{2}]$.
\end{theorem}

\begin{proof}
	Let $x,y,z$ be defined as in Definition \ref{divergent}. Then, there exists a \emph{fixed point} $a$ between $x$ and $y$ and a \emph{fixed point} $b$ between $y$ and $z$. Also, there exists a point $c$ between $a$ and $y$ such that $f(c) =b$. Let $J =[a,c]$, $K_1 = [c,y]$, $K_2 = [y,b]$ and $L = [b,z]$. Then, the interval $J$ $f$-\emph{covers} the interval $J,K_1$ and $K_2$; the interval $K_1$ $f$-\emph{covers} $L$ ; the interval $K_2$ $f$-\emph{covers} $L$ and the interval $L$ $f$-\emph{covers} $J, K_1, K_2$.  Consider the \emph{loop} $ \alpha :  J \to J \to \dots J \to K_2 \to L \to J \to K_1 \to L \to J \to K_1 \to L \to J \to K_1 \to L \to \dots L \to J \to K_1 \to L \to J $ consisting of  $q-3p-3$ $J's$ followed by $K_2$, then $p$ \emph{blocks} $L \to J \to K_1$, then the \emph{block} $ L \to J$. Let $Q$ be the \emph{periodic orbit} \emph{associated} with $\alpha$ (Lemma \ref{ALM2}). From the \emph{location} of the \emph{intervals}: $J, K_1, K_2$ and $L$, it is clear that the point of \emph{absolute minimum} $m_Q$ of $Q$,  lies in the interval $L$ and the point of \emph{absolute maximum} $M_Q$ of $Q$,  lies in $ K_1 \cup K_2$. Let us start at the \emph{left end point} $le(Q)$ of $Q$ and travel along the \emph{loop} $\alpha$. Clearly, $le(Q) \in J$. Construction of the \emph{loop} $\alpha$ reveals that,  during our \emph{motion} along the \emph{loop}: $\alpha$,  the only time we can possibly encounter an \emph{unfolding point} is when we \emph{pass} through the \emph{block} $L \to J \to K_1$ in $\alpha$; also,  each time we pass through the \emph{block} $L \to J \to K_1$, we confront \emph{one} ``new" \emph{unfolding point}. It follows,  that the number of \emph{unfolding points} of $Q$ is $p$. Thus, from Theorem \ref{unfolding:computation}, the \emph{unfolding number} of $Q$ is $\frac{p}{q}$ and the \emph{unfolding pair} of $Q$ is $(p,q)$. Since, $p$ and $q$ are arbitrary, $f$ has \emph{cycles} of all possible \emph{unfolding pairs} and hence the result follows. 
\end{proof}

 \section{Comparison with over-rotation numbers}\label{comparison:over:numbers}

From Theorem \ref{divergent:unfolding:interval},  it is evident that, for the purpose of studying \emph{co-existence} of \emph{cycles} with different \emph{unfolding pairs}, it is sufficient to study \emph{convergent patterns}, and hence only consider maps $f \in \mathcal{U}$. Further as discussed in the beginning of Section \ref{first:main:section}, it is sufficient to consider the case $m_{\pi} > M_{\pi}$,  where $m_{\pi} $ and $ M_{\pi}$ are the \emph{points} of \emph{absolute minimum} and \emph{absolute maximum} of $\pi$.  We begin this section by  providing an example to illustrate that, in general, \emph{over-rotation numbers} and \emph{unfolding numbers} are different entities (See Figure \ref{trimodal}).

\begin{example}\label{counter:example}
	Consider the \emph{trimodal pattern} characterised by the \emph{permutation} $\Pi = (1,2,5, 7,10,3,6,8,9,4,11) $. It has \emph{over-rotation number} $\frac{3}{11}$, but \emph{unfolding number} $\frac{5}{11}$ (See Figure \ref{trimodal}).
	
\end{example}

In the end, we now portray the existence of a peculiar \emph{patterns} for which \emph{over-rotation number} equals \emph{unfolding number}!

\begin{definition}
A  \emph{convergent pattern} $\pi$ is called \emph{sheer},  if it is \emph{convergent} and is \emph{decreasing} in $[M_{\pi}, m_{\pi}]$. 
\end{definition}

\begin{theorem}\label{overrotation:unfolding:equality}
	Let $\pi$ be a convergent pattern with over-rotation number $orn(\pi)$ and unfolding number $un(\pi)$. Then,   $orn(\pi) \leqslant un(\pi)$. Further, if  $\pi$ is sheer, $orn(\pi) = un(\pi)$.
\end{theorem}

\begin{proof}
 Let $P$ be any \emph{cycle} which \emph{exhibits} $\pi$ and let $f \in \mathcal{U}$ be the $P$-\emph{linear} map with  \emph{unique fixed point} $a_f$.   Let $m_f$ and $M_f$ be points of \emph{absolute minimum} and \emph{absolute maximum} of $f$. We assume $M_f < a_f < m_f$ and thus \emph{points} of $P$ to the \emph{left} of $a_f$ map to the \emph{right} of \emph{itself} while \emph{points} of $P$ to the \emph{right} of $a_f$ map to the \emph{left} of \emph{itself}. Let the \emph{over-rotation number} of $P$ be $\frac{p}{q}$, indicating that there are \emph{exactly} $p$ points to the \emph{right of the fixed point} $a_f$ which are mapped to the \emph{left} of $a_f$. Let these $p$ points be denoted as: $ x_1 > x_2 > x_3 > \dots > x_p$. It's easy to see that for each $ i \in \{1,2, \dots p\}$,  precisely one of  $f^2(x_i)$ and $f(x_i)$  is an \emph{unfolding point}:  let us denote this \emph{unfolding point} by $\xi_i$,  $\xi_i \in \{ f^2(x_i), f(x_i) \} $ for $ i \in \{1,2, \dots p\}$. This indicates that the number of \emph{unfolding points} of  $P$ is at-least $p$.   Theorem \ref{unfolding:computation},  now yields that \emph{unfolding number} of $P$ is at-least $\frac{p}{q}$,  establishing the first result.

Let us now assume that $\pi$ is \emph{sheer}. With this assumption, let us exhibit that number of  \emph{unfolding points} of $P$ is precisely $p$. By way of contradiction, let's entertain the possibility of  an \emph{unfolding point} $\eta$ of $P$ such that: $\eta \notin \{\xi_i : i =1, 2, \dots p \}$. We can then encounter \emph{four} potential \emph{cases}:

\begin{enumerate}
	\item $f(m_f) \leqslant M_f \leqslant f^{-2}(\eta) \leqslant a_f \leqslant f^{-1}(\eta) \leqslant m_f \leqslant f(M_f)$.
	
	\item $f(m_f) \leqslant M_f \leqslant f^{-2}(\eta) \leqslant  f^{-1}(\eta) \leqslant a_f \leqslant  m_f \leqslant f(M_f)$.
	
	\item $f(m_f) \leqslant M_f \leqslant f^{-1}(\eta) \leqslant a_f \leqslant f^{-2}(\eta)  \leqslant f(M_f)$.
	
	\item $f(m_f) \leqslant M_f  \leqslant a_f \leqslant f^{-1}(\eta)  \leqslant f^{-2}(\eta) \leqslant f(M_f)$ and $ f^{-1}(\eta) \leqslant m_f$
\end{enumerate}

We illustrate,  that all these \emph{four cases} leads us to a contradiction. Let's first examine the \emph{first} \emph{case}.  We claim $a_f \leqslant \eta \leqslant f^{-1}(\eta) $. Indeed,  $\eta \leqslant a_f$ means that $f^{-1}(\eta) \in \{x_i : i =1, 2, \dots p \}$,  and hence $\eta \in  \{f(x_i) : i =1, 2, \dots p \}$ contradicting the fact that $\eta \notin \{\xi_i : i =1, 2, \dots p \}$. Consequently, $a_f \leqslant \eta \leqslant f^{-1}(\eta) $. But,  this in turn means, $f$ first \emph{decreases} in the interval $[M_f, a_f]$, then \emph{increases}  in the interval $[a_f, f^{-1}(\eta)]$ and then \emph{decreases} in $[f^{-1}(\eta), m_f]$ which contradicts the fact that $\pi$ is a \emph{sheer pattern}.  

 In the \emph{second case}, since,  $\eta \geqslant f^{-1}(\eta)$, $f$ first \emph{decreases} in the \emph{interval} $[M_f, f^{-2}(\eta)]$, then \emph{increases} in the interval $[ f^{-2}(\eta),  f^{-1}(\eta)]$, and finally \emph{decreases} in $[f^{-1}(\eta), m_f]$, again contradicting the assumption that $\pi$ is a \emph{sheer pattern}.

 In the \emph{third case}, $f^{-2}(\eta) \in \{x_i : i =1, 2, \dots p\}$ and hence, $\eta \in \{f^2(x_i) : i =1, 2, \dots p \}$ again, contradicting the fact that $\eta \notin \{\xi_i : i =1, 2, \dots p \}$.  Lastly, the arguments for the \emph{fourth case} exactly mirror that of the \emph{first}. Thus, \emph{number} of \emph{unfolding points} of $P$ is precisely $p$ and hence from Theorem \ref{unfolding:computation}, \emph{unfolding number} of $P$ is $\frac{p}{q}$. 
 \end{proof}

Notice,  that every \emph{unimodal} or \emph{bimodal pattern} $\pi$ inherently qualifies as a \emph{sheer pattern}, leading us to the subsequent corollary:
\begin{cor}\label{overrotation:unfolding:equality:bimodal}
	If $\pi$ has modality less than or equal to $2$, then over-rotation number of $\pi$= unfolding number of $\pi$. 
\end{cor}

\begin{figure}[H]
	\caption{A \emph{trimodal pattern} having \emph{over-rotation number}: $\frac{3}{11}$ but \emph{unfolding number}: $\frac{5}{11}$.}
	\centering
	\includegraphics[width=0.6 \textwidth]{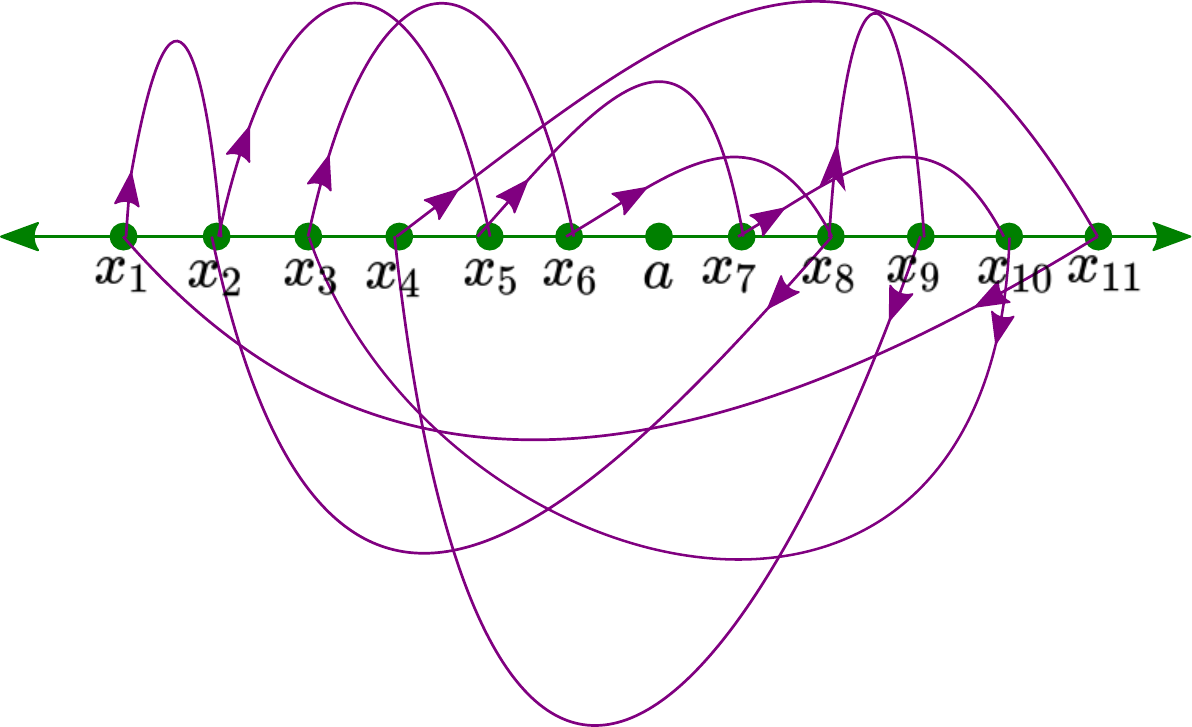}
	\label{trimodal}
\end{figure}

\end{document}